\documentclass[reqno,12pt]{amsart}
\usepackage{amsmath,amsthm,amssymb}
\usepackage{fullpage}
\usepackage{xcolor}
\usepackage{stmaryrd}
\begingroup
\catcode`[=\active \catcode`]=\active
\gdef[{\llbracket} \gdef]{\rrbracket}
\def\getdelim#1#2#3#4\relax{"#4}
\global\delcode`[=\expandafter\getdelim\llbracket\relax
\global\delcode`]=\expandafter\getdelim\rrbracket\relax
\endgroup
\newtheorem{thm}{Theorem}[section]
\newtheorem{lem}[thm]{Lemma}

\theoremstyle{definition}
\newtheorem{defn}[thm]{Definition}

\theoremstyle{remark} \numberwithin{equation}{section}

\begin{document}
\title[]{Symmetric positive solutions for a fractional singular integro-differential boundary value problem in presence of Caputo-Fabrizio fractional derivative}
\date{}
\author{Naseer Ahmad Asif}
\address{Department of Mathematics, University of Management and Technology, C-II Johar Town, 54770 Lahore, Pakistan}%
\email{naseerasif@yahoo.com}%
\keywords{Caputo and Fabrizio; Integro-differential equations; Positive solutions; Singular boundary value problems.}
\begin{abstract}

{We introduce the notion of Caputo-Fabrizio left and right derivatives. We present sufficient conditions for the existence of symmetric positive solutions for the following Caputo-Fabrizio fractional singular integro-differential boundary value problem
\begin{align*}
(2-\mu)\,{}^{CF}D_{0}^{\,\mu}x(t)+f(t,x(t))&=\left(\frac{\mu-1}{2-\mu}\right)^{2}
\begin{cases}
\int_{t}^{0}e^{-\frac{\mu-1}{2-\mu}(\tau-t)}\,x(\tau)d\tau,\hspace{0.4cm}&t\in(-1,0],\\
\int_{0}^{t}e^{-\frac{\mu-1}{2-\mu}(t-\tau)}\,x(\tau)d\tau,\hspace{0.4cm}&t\in[0,1),
\end{cases}\\
x(\pm1)=x'(0^{\pm})&=0,\hspace{5.55cm}\mu\in(1,2),
\end{align*}
where the nonlinearity $f:(-1,\,1)\times(0,\infty)\rightarrow\mathbb{R}$ is continuous and singular at $t=-1$, $t=1$ and $x=0$.}\end{abstract} \maketitle

\section{introduction}

Recently, the author presented the existence of symmetric positive solutions for the following Caputo fractional singular boundary value problems (SBVPs)
\begin{equation}\label{cbvp}\begin{split}
{}^{C}D_{0}^{\,\mu}x(t)+f(t,x(t))&=\omega\,x(t),\hspace{0.4cm}t\in(-1,\,1),\hspace{0.4cm}1<\mu\leq2,\\
x(\pm1)=x'(0^{\pm})&=0,
\end{split}\end{equation}
for $\omega=0$ \cite{naseer1}, and for $\omega>0$ \cite{naseer2}. Here ${}^{C}D_{0}^{\,\mu}x(t)={}^{C}D_{0^{+}}^{\,\mu}x(t)$ for $t\geq0$, ${}^{C}D_{0}^{\,\mu}x(t)={}^{C}D_{0^{-}}^{\,\mu}x(t)$ for $t\leq0$. Moreover, $f:(-1,\,1)\times(0,\infty)\rightarrow\mathbb{R}$ is continuous and singular at $t=-1$, $t=1$ and $x=0$. Here, ${}^{C}D_{0^{+}}^{\,\mu}$ and ${}^{C}D_{0^{-}}^{\,\mu}$, respectively, are Caputo fractional left and right derivatives of order $\mu$.

In this manuscript the author generalize the definition of Caputo-Fabrizio fractional derivative for an arbitrary order and present the definitions of Caputo-Fabrizio fractional left and right derivatives. Further, to utilize these notions, the existence results for the following Caputo-Fabrizio fractional singular integro-differential boundary value problem (SIDBVP)
\begin{equation}\label{mp}\begin{split}
(2-\mu)\,{}^{CF}D_{0}^{\,\mu}x(t)+f(t,x(t))&=\left(\frac{\mu-1}{2-\mu}\right)^{2}
\begin{cases}
\int_{t}^{0}e^{-\frac{\mu-1}{2-\mu}(\tau-t)}\,x(\tau)d\tau,\hspace{0.4cm}&t\in(-1,0]\\
\int_{0}^{t}e^{-\frac{\mu-1}{2-\mu}(t-\tau)}\,x(\tau)d\tau,\hspace{0.4cm}&t\in[0,1),
\end{cases}\\
x(\pm1)=x'(0^{\pm})&=0,\hspace{6.3cm}\mu\in(1,2),
\end{split}\end{equation}
are presented, where ${}^{CF}D_{0}^{\,\mu}x(t)={}^{CF}D_{0^{+}}^{\,\mu}x(t)$ for $t\geq0$, ${}^{CF}D_{0}^{\,\mu}x(t)={}^{CF}D_{0^{-}}^{\,\mu}x(t)$ for $t\leq0$. Moreover, $f:(-1,\,1)\times(0,\infty)\rightarrow\mathbb{R}$ is continuous and singular at $t=-1$, $t=1$ and $x=0$. Here, ${}^{CF}D_{0^{+}}^{\,\mu}$ and ${}^{CF}D_{0^{-}}^{\,\mu}$, respectively, are Caputo-Fabrizio fractional left and right derivatives of order $\mu$. We formulate sufficient conditions for the existence of symmetric positive solutions for Caputo-Fabrizio fractional SIDBVP \eqref{mp}. By a symmetric positive solution $x$ of Caputo-Fabrizio fractional SIDBVP \eqref{mp} we mean $x\in C[-1,\,1]$ satisfies \eqref{mp}, $x(t)=x(-t)$ for $t\in[-1,\,1]$ and $x(t)>0$ for $t\in(-1,\,1)$. For Caputo-Fabrizio fractional SIDBVP \eqref{mp}, two point boundary conditions (BCs) are $x(\pm1)=0$, whereas $x'(0^{\pm})=0$ are natural conditions followed by a function $x$ which is symmetric and concave on $[-1,\,1]$.

BVPs involving fractional order differentials have become an emerging area of recent research in science, engineering and mathematics, \cite{kilbas,lak1,lak3,miller,podlubny}. Applying results of nonlinear functional analysis and fixed point theory, many articles have been devoted to the study the existence of solutions for fractional order BVPs \cite{eloe,jiang,xu,zhang,zhanghu,zhao}. Further, many authors considered the SBVPs involving fractional derivatives, for details see \cite{afshari,agarwal,shahed,stanek,zhangmao}. Nonlinear IDBVPs has formulated an important area of research as these equations represents many continuum phenomena and arise in a various fields of science such as population dynamics, ecology, fluid mechanics, aerodynamics, etc \cite{ahmad,baleanu,chang,kosmatov,lak2,sun}. However, most of the previous results has been obtained for BVPs involving Riemann-Liouville and Caputo fractional derivatives. Further, most of the work has been devoted to fractional order left derivatives only. Therefore, study of Caputo-Fabrizio fractional SIDBVP \eqref{mp} is important due to presence of Caputo-Fabrizio fractional left and right derivatives.

To the author's best knowledge, \eqref{mp} is a new form of fractional differential equation. Further, no paper in the existing literature has formulated the existence of symmetric positive solutions for BVPs containing Caputo-Fabrizio fractional derivative. An important features of this manuscript is that the Green's function of the Caputo-Fabrizio fractional linear IDBVP corresponding to \eqref{mp} is symmetric. The manuscript is organized as follows. In Section \ref{pre}, we presents generalized definitions of Caputo-Fabrizio fractional left and right derivatives, some preliminary lemmas for the construction of Green's function. Some properties of Green's function are also presented. In Section \ref{main}, the existence of symmetric positive solutions has been established in Theorem \ref{mainth}.

\section{preliminaries}\label{pre}
We generalize the definition of Caputo-Fabrizio fractional derivative \cite{CF} and define the Caputo-Fabrizio fractional left and right derivatives as follows.

\begin{defn}
The Caputo-Fabrizio fractional left derivative of a function $x\in AC^{\,n}[0,\infty)$, $n\in\mathbb{N}$, of order $\mu\in(n-1,n)$ is defined as
\begin{equation}\label{cfdev}
{}^{CF}D_{0^{+}}^{\,\mu}\,x(t)=\frac{1}{n-\mu}\int_{0}^{t}e^{-\frac{\mu-n+1}{n-\mu}(t-\tau)}\,x^{(n)}(\tau)d\tau. \end{equation}
\end{defn}

\begin{defn}
The Caputo-Fabrizio fractional right derivative of a function $x\in AC^{\,n}(-\infty,0]$, $n\in\mathbb{N}$, of order $\mu\in(n-1,n)$ is defined as
\begin{equation}\label{cfdev}
{}^{CF}D_{0^{-}}^{\,\mu}\,x(t)=\frac{(-1)^{n}}{n-\mu}\int_{t}^{0}e^{-\frac{\mu-n+1}{n-\mu}(\tau-t)}\,x^{(n)}(\tau)d\tau.\end{equation}
\end{defn}

\begin{lem}\label{rightsol}
Let $\mu\in(1,2)$, $y:(-\infty,0]\rightarrow\mathbb{R}$ with $y(0)=0$. Then the Caputo-Fabrizio fractional integro-differential equation
\begin{align*}
(2-\mu)\,{}^{CF}D_{0^{-}}^{\,\mu}x(t)+y(t)&=\left(\frac{\mu-1}{2-\mu}\right)^{2}\int_{t}^{0}e^{-\frac{\mu-1}{2-\mu}(\tau-t)}x(\tau)d\tau,\hspace{0.4cm}t\leq0,
\end{align*} has general solution
\begin{align*}
x(t)=a_{1}\cosh[\frac{\mu-1}{2-\mu}t]+a_{2}\sinh[\frac{\mu-1}{2-\mu}t]
-\int_{t}^{0}e^{\frac{\mu-1}{2-\mu}(\tau-t)}\,y(\tau)d\tau,\hspace{0.4cm}t\leq0,
\end{align*}
where $a_{1},a_{2}\in\mathbb{R}$.
\end{lem}

\begin{lem}\label{leftsol}
Let $\mu\in(1,2)$, $y:[0,\infty)\rightarrow\mathbb{R}$ with $y(0)=0$. Then the Caputo-Fabrizio fractional integro-differential equation
\begin{align*}
(2-\mu)\,{}^{CF}D_{0^{+}}^{\,\mu}x(t)+y(t)&=\left(\frac{\mu-1}{2-\mu}\right)^{2}\int_{0}^{t}e^{-\frac{\mu-1}{2-\mu}(t-\tau)}x(\tau)d\tau,\hspace{0.4cm}t\geq0,
\end{align*} has general solution
\begin{align*}
x(t)=b_{1}\cosh[\frac{\mu-1}{2-\mu}t]+b_{2}\sinh[\frac{\mu-1}{2-\mu}t]
-\int_{0}^{t}e^{\frac{\mu-1}{2-\mu}(t-\tau)}\,y(\tau)d\tau,\hspace{0.4cm}t\geq0,
\end{align*}
where $b_{1},b_{2}\in\mathbb{R}$.
\end{lem}

\begin{lem}\label{lemir}
For $\mu\in(1,2)$, $y(0)=0$, the Caputo-Fabrizio fractional IDBVP
\begin{equation}\label{cfir}\begin{split}
(2-\mu)\,{}^{CF}D_{0}^{\,\mu}x(t)+y(t)&=\left(\frac{\mu-1}{2-\mu}\right)^{2}
\begin{cases}
\int_{t}^{0}e^{-\frac{\mu-1}{2-\mu}(\tau-t)}x(\tau)d\tau,\hspace{0.4cm}&t\in(-1,0]\\
\int_{0}^{t}e^{-\frac{\mu-1}{2-\mu}(t-\tau)}x(\tau)d\tau,\hspace{0.4cm}&t\in[0,1),
\end{cases}\\
x(\pm1)=x'(0^{\pm})&=0,
\end{split}\end{equation}
has integral representation
\begin{equation}\label{ir}
x(t)=\int_{\Lambda}G(t,\tau)y(\tau)d\tau,\hspace{0.4cm}t\in\Lambda=[-1,0],[0,1],
\end{equation}
where
\begin{equation}\label{gf}
G(t,\tau)=\begin{cases}
\frac{\cosh[\frac{\mu-1}{2-\mu}t]}{\cosh[\frac{\mu-1}{2-\mu}]}e^{\frac{\mu-1}{2-\mu}(1+\tau)},\,&-1\leq\tau\leq t\leq0,\\
\frac{\cosh[\frac{\mu-1}{2-\mu}t]}{\cosh[\frac{\mu-1}{2-\mu}]}e^{\frac{\mu-1}{2-\mu}(1+\tau)}-e^{\frac{\mu-1}{2-\mu}(\tau-t)},\,&-1\leq t\leq\tau\leq0,\\
\frac{\cosh[\frac{\mu-1}{2-\mu}t]}{\cosh[\frac{\mu-1}{2-\mu}]}e^{\frac{\mu-1}{2-\mu}(1-\tau)}-e^{\frac{\mu-1}{2-\mu}(t-\tau)},\,&\,\,\,\,\,0\leq\tau\leq t\leq1,\\
\frac{\cosh[\frac{\mu-1}{2-\mu}t]}{\cosh[\frac{\mu-1}{2-\mu}]}e^{\frac{\mu-1}{2-\mu}(1-\tau)},\,&\,\,\,\,\,0\leq t\leq\tau\leq1.
\end{cases}\end{equation}
\end{lem}

\begin{proof}
In view of \eqref{cfir}, we have
\begin{align*}
(2-\mu)\,{}^{CF}D_{0^{-}}^{\,\mu}x(t)+y(t)&=\left(\frac{\mu-1}{2-\mu}\right)^{2}
\int_{t}^{0}e^{-\frac{\mu-1}{2-\mu}(\tau-t)}x(\tau)d\tau,\hspace{0.4cm}t\in(-1,0],\\
(2-\mu)\,{}^{CF}D_{0^{+}}^{\,\mu}x(t)+y(t)&=\left(\frac{\mu-1}{2-\mu}\right)^{2}
\int_{0}^{t}e^{-\frac{\mu-1}{2-\mu}(t-\tau)}x(\tau)d\tau,\hspace{0.4cm}t\in[0,1),
\end{align*}
which in view of Lemma \ref{rightsol} and Lemma \ref{leftsol}, leads to
\begin{align*}
x(t)&=a_{1}\cosh[\frac{\mu-1}{2-\mu}t]+a_{2}\sinh[\frac{\mu-1}{2-\mu}t]
-\int_{t}^{0}e^{\frac{\mu-1}{2-\mu}(\tau-t)}\,y(\tau)d\tau,\hspace{0.4cm}t\in[-1,0],\\
x(t)&=b_{1}\cosh[\frac{\mu-1}{2-\mu}t]+b_{2}\sinh[\frac{\mu-1}{2-\mu}t]
-\int_{0}^{t}e^{\frac{\mu-1}{2-\mu}(t-\tau)}\,y(\tau)d\tau,\hspace{0.4cm}t\in[0,\,1],
\end{align*}
which on employing the BCs \eqref{cfir}, reduces to
\begin{align*}
x(t)&=\frac{\cosh[\frac{\mu-1}{2-\mu}t]}{\cosh[\frac{\mu-1}{2-\mu}]}\int_{-1}^{0}e^{\frac{\mu-1}{2-\mu}(1+\tau)}\,y(\tau)d\tau
-\int_{t}^{0}e^{\frac{\mu-1}{2-\mu}(\tau-t)}\,y(\tau)d\tau,\hspace{0.4cm}t\in[-1,0],\\
x(t)&=\frac{\cosh[\frac{\mu-1}{2-\mu}t]}{\cosh[\frac{\mu-1}{2-\mu}]}\int_{0}^{1}e^{\frac{\mu-1}{2-\mu}(1-\tau)}\,y(\tau)d\tau
-\int_{0}^{t}e^{\frac{\mu-1}{2-\mu}(t-\tau)}\,y(\tau)d\tau,\hspace{0.4cm}t\in[0,\,1],
\end{align*}
which is equivalent to \eqref{ir}.
\end{proof}

In the following Lemma \ref{gbound}, we present some properties of the Green's function \eqref{gf}.

\begin{lem}\label{gbound}
The Green's function \eqref{gf} satisfies
\begin{itemize}
\item[(1).] $G:[-1,\,1]\times[-1,\,1]\rightarrow[0,\infty)$ is continuous and positive on $(-1,\,1)\times(-1,\,1)$.
\item[(2).] $G(t,\tau)=G(-t,-\tau)$ for all $(t,\tau)\in[-1,\,1]\times[-1,\,1]$.
\item[(3).] $G(t,\tau)\leq1$ for all $(t,\tau)\in[-1,\,1]\times[-1,\,1]$.
\end{itemize}
\end{lem}

\begin{proof}
The proof is obvious.
\end{proof}

\section{Main Result}\label{main}
We make the following assumptions.
\begin{itemize}
\item[(A1).] For each $x>0$, $f(0,x)=0$, $f(t,x)=f(-t,x)$ for all $t\in(-1,\,1)$. There exist $q\in C[0,1)$, $u\in C(0,\infty)$ decreasing, and $v\in C[0,\infty)$ increasing such that
\begin{align*}|f(t,x)|\leq q(|t|)(u(x)+v(x)),\hspace{0.4cm}t\in(-1,\,1),\hspace{0.4cm}x\in(0,\infty),\end{align*}
\item[(A2).] There exist a constant $R\geq\sigma_{R}(0)$ such that, for $t\in(-1,\,1)$ and $x\in(0,R]$, $f(t,x)\geq\psi_{R}(|t|)$, where $\psi_{R}:[0,\,1)\rightarrow[0,\infty)$, $\sigma_{R}(t):=\int_{0}^{1}G(t,\tau)\psi_{R}(\tau)d\tau<\infty$.
Moreover,
\begin{align*}
\int_{0}^{1}q(t)dt<\infty,\text{ and }\int_{0}^{1}q(t)u(\sigma_{R}(t))dt<\infty.
\end{align*}
\begin{align*}
\frac{R}{\left(1+\frac{v(R)}{u(R)}\right)\int_{0}^{1}q(t)u(\sigma_{R}(t))dt}>1.
\end{align*}
\end{itemize}

In view of $(A2)$, choose $\varepsilon>0$ such that
\begin{equation}\label{eps}
\frac{R-\varepsilon}{\left(1+\frac{v(R)}{u(R)}\right)\int_{0}^{1}q(t)u(\sigma_{R}(t))dt}\geq1.
\end{equation}
Let $X=\{x:x\in C[-1,\,1],\,x(t)=x(-t)\text{ for }t\in[-1,\,1]\}$. For each $m\in\mathbb{N}$ with $\frac{1}{m}<\varepsilon$, define $T_{m}:X\rightarrow X$ by
\begin{equation}\label{mapt}
T_{m}x(t)=\int_{\Lambda}G(t,\tau)\,f\left(\tau,\min\{\max\{x(\tau)+\frac{1}{m},\frac{1}{m}\},R\}\right)d\tau,\hspace{0.4cm}t\in\Lambda.
\end{equation}

\begin{thm}\label{mainth}
Assume that $(A1)$ and $(A2)$ hold. Then the Caputo-Fabrizio fractional SIDBVP \eqref{mp} has a symmetric positive solution.
\end{thm}

\begin{proof}
In view of $(A1)$ and Schauder's fixed point theorem the map $T_{m}$ defined by \eqref{mapt} has a fixed point $x_{m}\in X$. Thus
\begin{equation}\label{fp}
x_{m}(t)=\int_{\Lambda}G(t,\tau)f\left(\tau,\min\{\max\{x_{m}(\tau)+\frac{1}{m},\frac{1}{m}\},R\}\right)d\tau,\hspace{0.4cm}t\in\Lambda,
\end{equation}
which in view of $(A2)$ and Lemma \ref{gbound}, leads to
\begin{equation}\label{lb}\begin{split}
x_{m}(t)&\geq\int_{\Lambda}G(t,\tau)\,\psi_{R}(|\tau|)d\tau,\hspace{0.4cm}t\in\Lambda,\\
&=\sigma_{R}(|t|),\hspace{0.4cm}t\in[-1,\,1].
\end{split}\end{equation}
Also \eqref{fp} in view of Lemma \ref{gbound}, $(A1)$, \eqref{lb} and \eqref{eps}, leads to
\begin{align*}\label{fp}
\begin{split}
x_{m}(t)&=\int_{\Lambda}G(t,\tau)f\left(\tau,\min\{\max\{x_{m}(\tau)+\frac{1}{m},\frac{1}{m}\},R\}\right)d\tau,\hspace{0.4cm}t\in\Lambda,\\
&\leq\int_{\Lambda}G(t,\tau)q(|\tau|)u\left(\min\{\max\{x_{m}(\tau)+\frac{1}{m},\frac{1}{m}\},R\}\right)\\
&\hspace{0.3cm}\left(1+\frac{v(\min\{\max\{x_{m}(\tau)+\frac{1}{m},\frac{1}{m}\},R\})}{u(\min\{\max\{x_{m}(\tau)+\frac{1}{m},\frac{1}{m}\},R\})}\right)d\tau,\hspace{1.8cm}t\in\Lambda,\\
&\leq\int_{\Lambda}G(t,\tau)q(|\tau|)u(\sigma_{R}(|\tau|))\left(1+\frac{v(R)}{u(R)}\right)d\tau,\hspace{1.9cm}t\in\Lambda,\\
&\leq \left(1+\frac{v(R)}{u(R)}\right)\int_{0}^{1}q(\tau)u(\sigma_{R}(\tau))d\tau,\hspace{0.4cm}t\in[-1,\,1],
\end{split}
\end{align*}
which in view of \eqref{eps}, leads to
\begin{equation}\label{ub}\begin{split}
x_{m}(t)\leq R-\varepsilon,\hspace{0.4cm}t\in[-1,\,1].
\end{split}\end{equation}

Consequently, in view of \eqref{lb} and \eqref{ub}, we have

\begin{align*}
\sigma_{R}(|t|)\leq x_{m}(t)<R,\hspace{0.4cm}t\in[-1,\,1],
\end{align*}
which shows that the sequence $\{x_{n}\}_{n=m}^{\infty}$ is uniformly bounded on $[-1,1]$. Moreover, since $G(t,\tau)$ is uniformly continuous on $[-1,\,1]\times[-1,\,1]$, the sequence $\{x_{n}\}_{n=m}^{\infty}$ is equicontinuous on $[-1,\,1]$. Thus by Arzel\`{a}-Ascoli Theorem the sequence $\{x_{n}\}_{n=m}^{\infty}$ is relatively compact and consequently there exist a subsequence $\{x_{n_{k}}\}_{k=1}^{\infty}$ converging uniformly to $x\in X$. Moreover,
\begin{align*}
x_{n_{k}}(t)=\int_{\Lambda}G(t,\tau)f\left(\tau,x_{n_{k}}(\tau)\right)d\tau,\hspace{0.4cm}t\in\Lambda,
\end{align*}
as $k\rightarrow\infty$, we obtain
\begin{equation}\label{intsol}
x(t)=\int_{\Lambda}G(t,\tau)f\left(\tau,x(\tau)\right)d\tau,\hspace{0.4cm}t\in\Lambda,
\end{equation}
which in view of Lemma \ref{lemir}, leads to
\begin{align*}
(2-\mu)\,{}^{CF}D_{0}^{\,\mu}x(t)+f(t,x(t))&=\left(\frac{\mu-1}{2-\mu}\right)^{2}
\begin{cases}
\int_{t}^{0}e^{-\frac{\mu-1}{2-\mu}(\tau-t)}\,x(\tau)d\tau,\hspace{0.4cm}&t\in(-1,0]\\
\int_{0}^{t}e^{-\frac{\mu-1}{2-\mu}(t-\tau)}\,x(\tau)d\tau,\hspace{0.4cm}&t\in[0,1),
\end{cases}\\
x(\pm1)=x'(0^{\pm})&=0.
\end{align*}
Also, ${}^{CF}D_{0}^{\,\mu}\,x\in C(-1,\,1)$. Further, from \eqref{intsol} in view of $(A2)$ and Lemma \ref{gbound}, we have
\begin{align*}
x(t)\geq\sigma_{R}(|t|),\hspace{0.4cm}t\in[-1,\,1],
\end{align*}
which shows that $x(t)>0$ for $t\in(-1,\,1)$. Hence $x\in C[-1,\,1]$ with ${}^{CF}D_{0}^{\,\mu}\,x\in C(-1,\,1)$ is a symmetric positive solution of the Caputo-Fabrizio fractional SIDBVP \eqref{mp}.
\end{proof}


\begin{thebibliography}{99}

\bibitem{afshari}H. Afshari, H. Marasi, H. Aydi, Existence and uniqueness of positive solutions for boundary value problems of fractional differential equations, Filomat 31:9(2017) 2675-2682.

\bibitem{agarwal}R.P. Agarwal, D. O'Regan, S. Stanek, Positive solutions for Dirichlet problems of singular nonlinear fractional differential equations, J. Math. Anal. Appl. 371 (2010) 57-68.

\bibitem{ahmad}B. Ahmad, J.J. Nieto, Existence results for nonlinear boundary value problems of fractional integro-differential equations with integral boundary conditons, Boundary Value Problems (2009) 1-11.

\bibitem{naseer1}arXiv:1904.07109

\bibitem{naseer2}Naseer Ahmad Asif, Existence of symmetric positive solutions for a Caputo fractional singular boundary value problem, presented at ``6th International IFS and Contemporary Mathematics Conference", June 7-10, 2019, Mersin University, Mersin, Turkey, (submitted for publication).

\bibitem{baleanu}D. Baleanu, A. Mousalou, S. Rezapour, On the existence of solutions for some infinite coefficient-symmetric Caputo-Fabrizio fractional integro-differential equations, Bound. Value Probl. 2017 (2017) 9 pages.

\bibitem{CF}M. Caputo, M. Fabrizio, A New Definition of Fractional Derivative without Singular Kernel, Progr. Fract. Differ. Appl. 1:2, (2015) 73-85.

\bibitem{chang}Y.K. Chang, V. Kavitha, M. Mallika Ajunnan, Existence and uniqueness of mild solutions to semilinear integro-differential equation of fractional order, Nonlinear Analysis 71 (2009) 5551-5559.

\bibitem{eloe}P.W. Eloe, J.W. Lyons, J.T. Neugebauer, An ordering on Green's functions for a family of two-point boundary value problems for fractional differential equations, Commun. Appl. Anal. 19 (2015) 453-462.

\bibitem{jiang}W. Jiang, The existence of solutions to boundary value problems of fractional differential equations at resonance, Nonlinear Analysis, 74 (2011) 1987-1994.

\bibitem{kilbas} A.A. Kilbas, H.M. Srivastava, J.J. Trujillo, Theory and Applications of Fractional Differential Equations, North-Holland Mathematics Studies, vol. 204, Elsevier Science B.V., Amsterdam, 2006.

\bibitem{kosmatov}N. Kosmatov, Integral equations and initial value problems for nonlinear differential equations of fractional order, Nonlinear Analysis, 70 (2009) 2521-2529.

\bibitem{lak1} V. Lakshmikantham, Theory of fractional functional differential equations, Nonlinear Analysis 69 (2008) 3337-3343.

\bibitem{lak2}V. Lakshmikantham, M.R.M. Rao, Theory of Integro-Differential Equations. Stability and Control: Theory, Methods and Applications, vol. 1. Gordon and Breach Science Publishers, Lausanne, 1995.

\bibitem{lak3}V. Lakshmikantham, A.S. Vatsala, Basic theory of fractional differential equations, Nonlinear Analysis 69 (2008) 2677-2682.

\bibitem{losada}J. Losada, J. J. Nieto, Properties of a New Fractional Derivative without Singular Kernel, Progr. Fract. Differ. Appl., 1 (2015), 87-92.

\bibitem{miller}K. S. Miller, B. Ross, An Introduction to Fractional Calculus and Fractional Differential Equations, John Wiley and Sons, New York, 1993.

\bibitem{podlubny}I. Podlubny, Fractional Differential Equations. An Introduction to Fractional Derivatives, Fractional Differential Equations, to Methods of Their Solutions and Some of Their Applications, Mathematics in Science and Engineering, vol. 198, Academic Press, San Diego, 1999.

\bibitem{shahed}M. El-Shahed, Positive solutions for boundary value problem of nonlinear fractional differential equation, Abstr. Appl. Anal. 2007 (2007) 8 pages. Article ID 10368.

\bibitem{stanek}S. Stanek, The existence of positive solutions of singular fractional boundary value problems, Comput. Math. Appl. 62 (2011) 1379-1388.

\bibitem{sun}Y. Sun, Positive solutions of Sturm-Liouville boundary value problems for singular nonlinear second-order impulsive integro-differential equation in Banach spaces, Bound. Value Probl. 2012 (2012), 18 pages.

\bibitem{xu}X. Xu, D. Jiang, C. Yuan, Multiple positive solutions for the boundary value problem of a nonlinear fractional differential equation, Nonlinear Anal. 71 (2009) 4676-4688.

\bibitem{zhang}S.-Q. Zhang, Positive solutions for boundary value problems of nonlinear fractional differential equations. Electron. J. Differential Equations 36 (2006) 1-12.

\bibitem{zhangmao}X. Zhang, C. Mao, Y. Wu, H. Su, Positive solutions of a singular nonlocal fractional order differetial system via Schauder's fixed point theorem, Abstr. Appl. Anal. (2014) Art. ID 457965.

\bibitem{zhanghu}S. Zhanga, L. Hu, S. Sun, The uniqueness of solution for initial value problems for fractional differential equation involving the Caputo-Fabrizio derivative, J. Nonlinear Sci. Appl. 11 (2018) 428-436.

\bibitem{zhao}Y. Zhao, S. Sun, Z. Han, Q. Li, The existence of multiple positive solutions for boundary value problems of nonlinear fractional differential equations, Commun. Nonlinear Sci. Numer. Simul. 16 (2011) 2086-2097.

\end{thebibliography}
\end{document}